\documentclass{amsart}
\usepackage{latexsym, amsfonts, amsmath, amssymb, amsthm, cite}

\newcommand{\C} {{\mathbb C}}                              %complex
\newcommand{\R} {{\mathbb R}}                              %reals
\newcommand{\Z} {{\mathbb Z}}                              %integers

\newcommand{\F} {{\mathbb F}}
\newtheorem{thm}{Theorem}[section]
\newtheorem{prop}[thm]{Proposition}
\newtheorem{cor}[thm]{Corollary}
\newtheorem{lem}[thm]{Lemma}

\begin{document}

\title{A Density Version of the Vinogradov Three Primes Theorem}
\author{Xuancheng Shao}
\address{Department of Mathematics \\ Stanford University \\
450 Serra Mall, Bldg. 380\\ Stanford, CA 94305-2125}
\email{xshao@math.stanford.edu}

\begin{abstract}
We prove that if $A$ is a subset of the primes, and the lower
density of $A$ in the primes is larger than $5/8$, then all
sufficiently large odd positive integers can be written as the sum
of three primes in $A$. The constant $5/8$ in this statement is the
best possible.
\end{abstract}

\maketitle

\section{Introduction}

In this paper, we study the classical problem of writing a positive
integer as sum of primes. The famous Goldbach conjecture says that
every even positive integer at least $4$ is the sum of two primes.
This problem is considered to be out of the scope of current
techniques. However, the ternary Goldbach problem of writing
positive integers as sums of three primes is much more tractable. In
1937, Vinogradov \cite{vinogradov1937} proved that all sufficiently
large odd positive integers can be written as sums of three primes.
See Chapter 8 of \cite{nathanson} for a classical proof using circle
method, and \cite{heath-brown} for an alternate proof. Subsequent
works have been done to refine what ``sufficiently large'' means.
Very recently it was announced in \cite{helfgott2013} that
Vinogradov's theorem is true for {\em all} odd positive integers. We
refer the reader to \cite{helfgott2012,helfgott2013} for this
exciting news and the history in this numerical regard.

The purpose of this paper is to prove a density version of the Vinogradov's three prime theorem. More precisely, let $\mathcal{P}$ be the set of all primes. For a subset $A\subset\mathcal{P}$, the lower density of $A$ in $\mathcal{P}$ is defined by
\[ \underline{\delta}(A)=\liminf_{N\rightarrow\infty}\frac{|A\cap [1,N]|}{|\mathcal{P}\cap [1,N]|}. \]
We would like to show that all sufficiently large odd positive integers can be written as sums of three primes in $A$, provided that $\underline{\delta}(A)$ is larger than a certain threshold.

\begin{thm}\label{thm:main}
Let $A\subset\mathcal{P}$ be a subset of the primes with
$\underline{\delta}(A)>5/8$. Then for sufficiently large odd
positive integers $N$, there exist $a_1,a_2,a_3\in A$ with
$N=a_1+a_2+a_3$.
\end{thm}

The constant $5/8$ in Theorem \ref{thm:main} cannot be improved. In fact, we may take
\[ A=\{p\in\mathcal{P}:p\equiv 1,2,4,7,13\text{ mod }15\}. \]
Then $\underline{\delta}=5/8$. It is not hard to see that if $N\equiv 14$ (mod $15$) then $N$ cannot be written as sum of three elements of $A$.

Our study of this problem is motivated by the work of Li and Pan \cite{li2007}. They proved the following asymmetric version of Theorem \ref{thm:main}:

\begin{thm}\label{thm:asym}
Let $A_1,A_2,A_3\subset\mathcal{P}$ be three subsets of the primes
with
$\underline{\delta}(A_1)+\underline{\delta}(A_2)+\underline{\delta}(A_3)>2$.
Then for sufficiently large odd positive integers $N$, there exists
$a_i\in A_i$ ($i=1,2,3$) with $N=a_1+a_2+a_3$.
\end{thm}

Their result is sharp in two different ways. One can either take
$A_1=A_2=\mathcal{P}$ and $A_3=\emptyset$, or take
$A_1=\mathcal{P}\setminus\{3\}$ and
$A_2=A_3=\{p\in\mathcal{P}:p\equiv 1(\text{ mod }3)\}$. Note that in
these two examples, the three sets $A_1,A_2,A_3$ are very distinct
from each other. Theorem \ref{thm:asym} immediately implies a weaker
version of Theorem \ref{thm:main}, with $5/8$ replaced by $2/3$.

In the examples considered above, the sumset
$A+A+A=\{a_1+a_2+a_3:a_1,a_2,a_3\in A\}$ fails to cover all
sufficiently large odd integers due to local obstructions. It is
natural to ask what the threshold of $\underline{\delta}(A)$ becomes
if local obstructions are excluded. We prove the following result in
this direction. For any positive integer $W$ and any reduced residue
$b\pmod W$, the lower density of $A$ in the residue class $b\pmod W$
is defined by
\[
\underline{\delta}(A;b,W)=\liminf_{N\rightarrow\infty}\frac{\#\{p\in
A\cap [1,N]:p\equiv b\pmod W\}}{\#\{p\in\mathcal{P}\cap
[1,N]:p\equiv b\pmod W\}}.
\]

\begin{thm}\label{thm:main2}
For any $\delta>0$ there is a positive integer $W=W(\delta)$ such
that the following statement holds. Let $A\subset\mathcal{P}$ be a
subset of the primes satisfying the density conditions
\begin{equation}\label{eq:mod3}
2\underline{\delta}(A;1,3)+\underline{\delta}(A;2,3)\geq 3/2+\delta, \ \
2\underline{\delta}(A;2,3)+\underline{\delta}(A;1,3)\geq 3/2+\delta,
\end{equation}
as well as the local conditions
\begin{equation}\label{eq:modW}
\underline{\delta}(A;b,W)\geq\delta
\end{equation}
for any reduced residue $b\pmod W$. Then for sufficiently large odd
positive integers $N$, there exist $a_1,a_2,a_3\in A$ with
$N=a_1+a_2+a_3$.
\end{thm}

In the case when
$\underline{\delta}(A;1,3)=\underline{\delta}(A;2,3)$, the
assumption on the density of $A$ becomes
$\underline{\delta}(A)>1/2$. We will see later that the conditions
\eqref{eq:mod3} and \eqref{eq:modW} naturally come from the
hypotheses in Proposition \ref{prop:main2} below. It is an
interesting question whether the condition
$\underline{\delta}(A)>1/2$ can be relaxed (maybe assuming stronger
local conditions). One nice consequence of breaking the density
$1/2$ barrier is that it opens the door for giving a new proof of
Vinogradov's theorem without using the theory of $L$-functions. This
will be carried out in a forthcoming paper \cite{shao}.

The proof of Theorem \ref{thm:main2} is by using the so-called
transference principle in additive combinatorics, which was used to
solve many other problems of similar nature (for example, see
\cite{green2005} and \cite{green2006} for the proof of Roth's
theorem in the primes). In our setting of representation of integers
as sums of three primes, this method was worked out in Section 3 of
\cite{li2007}. We reproduce this argument in Section \ref{sec:trans}
for the sake of completeness, and and we also hope to outline the
main ingredients of the transference principle without too many
technical details.

The main innovation of the current paper is the first part of the
proof of Theorem \ref{thm:main}, which asserts that $A+A+A$ must
cover all residue classes modulo $m$ for any odd $m$, provided that
the density of $A$ in $\Z_m^*$ is greater than $5/8$. The precise
statement is as follows.

\begin{prop}\label{prop:main}
Let $m$ be an odd squarefree positive integer. Let $f:\Z_m^*\rightarrow [0,1]$ be an arbitrary function with average larger than $5/8$:
\[ \sum_{x\in\Z_m^*}f(x)>\tfrac{5}{8}\phi(m). \]
Then for any $x\in\Z_m$, there exists $a_1,a_2,a_3\in\Z_m^*$ with
$x=a_1+a_2+a_3$, such that
\[ f(a_1)f(a_2)f(a_3)>0,\quad f(a_1)+f(a_2)+f(a_3)>\tfrac{3}{2}. \]
\end{prop}

In particular, we have:

\begin{cor}\label{cor:main}
Let $m$ be an odd squarefree positive integer. Let $A\subset\Z_m^*$ be a subset with $|A|>\tfrac{5}{8}\phi(m)$. Then $A+A+A=\Z_m$.
\end{cor}

Note that the constant $3/2$ in Proposition \ref{prop:main} cannot
be made any larger. In fact, we may take $m=3$ and $f(1)=1$,
$f(2)=1/4$. Then for $x=2$, we must have
$\{a_1,a_2,a_3\}=\{1,2,2\}$, and thus $f(a_1)+f(a_2)+f(a_3)=3/2$.

The proof of Proposition \ref{prop:main} consists of two parts. The
first part of the proof is an inductive argument that reduces the
problem to the case $m=15$ (see Proposition \ref{prop:5/8} below).
In the second part, a robust version of Proposition \ref{prop:main}
with $m=15$ is proved (see Proposition \ref{prop:15} below). The
conclusion of Proposition \ref{prop:5/8} is tailored so that it
matches the assumption of Proposition \ref{prop:15}. In comparison,
Li and Pan used a less involved inductive argument to prove an
analogous local result (Theorem 1.2 in \cite{li2007}).

We remark that if $m$ is prime, Corollary \ref{cor:main} is then a
simple consequence of the Cauchy-Davenport-Chowla inequality, which
asserts that if $A,B,C$ are subsets of $\Z_p$ for prime $p$, then
\[ |A+B+C|\geq\min(|A|+|B|+|C|-2,p). \]
See Theorem 5.4 of \cite{tao-vu} for a proof. In the other extreme
when $m$ is highly composite, the situation is quite different, as
$|A|$ is much smaller than $m$, and thus one has to use the
structure of the set of reduced residues $\Z_m^*$ in $\Z_m$.

Under the local assumptions in Theorem \ref{thm:main2}, however, a
result such as Proposition \ref{prop:main} is much easier. We state
such a result here, to be compared with Proposition \ref{prop:main}.

\begin{prop}\label{prop:main2}
Let $m$ be an odd squarefree positive integer with $3\mid m$. Let
$f:\Z_m^*\rightarrow (0,1]$ be an arbitrary function with average
$\alpha$. For $i=1,2$, let $\alpha_i$ be the average of $f$ over
those reduced residues $r\pmod m$ with $r\equiv i\pmod 3$. Suppose
that $2\alpha_1+\alpha_2>3/2$ and $2\alpha_2+\alpha_1>3/2$. Then for
any $x\in\Z_m$, there exists $a_1,a_2,a_3\in\Z_m^*$ with
$x=a_1+a_2+a_3$, such that
\[ f(a_1)+f(a_2)+f(a_3)>\tfrac{3}{2}. \]
\end{prop}

\begin{proof}
Write $m=3m'$ and use the isomorphism $\Z_m\cong\Z_3\oplus\Z_{m'}$.
Suppose that $x=(i,x')$ for some $i\in\{0,1,2\}$ and $x'\in\Z_{m'}$.
A straightforward case by case analysis shows that there exists
$i_1,i_2,i_3\in\{1,2\}$ with $i\equiv i_1+i_2+i_3\pmod 3$ such that
\[ \alpha_{i_1}+\alpha_{i_2}+\alpha_{i_3}>\tfrac{3}{2}. \]
For $k\in\{1,2,3\}$, define a function $f_k':\Z_{m'}^*\rightarrow
(0,1]$ by setting $f_k'(r')=f(i_k,r')$ for each $r'\in\Z_{m'}^*$.
Since the average of $f_k'$ is $\alpha_{i_k}$, the sum of the
averages of $f_k'$ ($k=1,2,3$) is larger than $3/2$. It then follows
from Theorem 1.2 of \cite{li2007} that there exists
$a_1',a_2',a_3'\in\Z_{m'}^*$ with $x'\equiv a_1'+a_2'+a_3'\pmod{m'}$
such that
\[ f_1'(a_1')+f_2'(a_2')+f_3'(a_3')>\tfrac{3}{2}. \]
The proof is then completed by taking $a_k=(i_k,a_k')$ ($k=1,2,3$).

\end{proof}

The rest of the paper is organized as follows. Proposition
\ref{prop:main} will be proved in Section \ref{sec:loc}, using an
elementary lemma established in Section \ref{sec:lem}. Then in
Section \ref{sec:trans}, we recall the transference principle in our
setting and use it to deduce Theorem \ref{thm:main} as well as
Theorem \ref{thm:main2}.

\vspace{5 mm} \textbf{Acknowledgement.} The author would like to
express his gratitude to Ben Green for suggesting this problem to
him, to his advisor, Kannan Soundararajan, for helpful comments on
exposition, and to the anonymous referees for suggestions and
remarks which greatly improved the paper.

%%%%%%%%%%%%%%%%%%%%%%%%%%%%%%%%%%% LEMMATA %%%%%%%%%%%%%%%%%%%%%%%%%%%%%%%%%

\section{Lemmata}\label{sec:lem}

\begin{lem}\label{lem:seq_sym}
Let $n\ge 6$ be an even positive integer.  Let $a_0 \ge a_1 \ge \ldots \ge a_{n-1}$
be a decreasing sequence of numbers in $[0,1]$, and let $A$ denote their average
$A= \frac 1n \sum_{j=0}^{n-1} a_j$.  Suppose that for all triples $(i,j,k)$ with $0\le i, j , k \le n-1$
and $i+j+k \ge n$ we have
$$
a_ia_j +a_j a_k + a_k a_i \le \tfrac 58 (a_i+a_j+a_k).
$$
Then the average value $A$ is bounded by $\frac 58$.
\end{lem}
\begin{proof}  We make the change of variables $x_i = \frac{16}{5} a_i -1$.  Then
the variables $x_i$ form a decreasing sequence of numbers in $[-1,11/5]$.   Our condition on
the sequence $a_i$ becomes, with a little calculation:  for all triples $(i,j,k)$ with
$0\le i, j, k \le n-1$ and $i+j+k \ge n$,
\begin{equation}
\label{eq:seq_sym_1}
x_i x_j + x_j x_k + x_k x_i \le 3.
\end{equation}
If $X$ denotes the average value $\frac 1n \sum_{i=0}^{n-1} x_i = \frac{16}{5} A -1$,
then our goal is to show that $X \le 1$.  Suppose instead that $X>1$, and
so in particular $x_0 > 1$.

Write $n=2m$ so that $m\ge 3$ is an integer, and define
$$
S_0 =\sum_{i=0}^{m-1} x_i, \qquad \text{and} \qquad S_1=\sum_{i=m}^{n-1} x_i,
$$
so that $S_0+S_1 = n X> n$.  Since $S_0 \le 2.2 m$, it follows that $S_1 \ge -0.2 m$.

Let ${\mathcal M}$ denote the set of triples $(i,j,k)$ with $0\le i, j \le m-1$ and $m\le k\le n-1$ with
$i+j+k \equiv 0\pmod m$.  Given $i$ and $j$ there is a unique $k$ with $(i,j,k) \in {\mathcal M}$.
Thus $|{\mathcal M}| = m^2$, and note that all triples in $\mathcal{M}$ except for $(0,0,m)$ satisfy $i+j+k \ge 2m
=n$.  Summing the condition \eqref{eq:seq_sym_1} over all triples in ${\mathcal M}$ except
for $(0,0,m)$ we get that
$$
\sum_{(i,j,k)\in{\mathcal M}} (x_ix_j +x_jx_k +x_kx_i) - (x_0^2 +2x_0x_m) \le 3(m^2-1).
$$
If we fix two of the variables $i$, $j$, $k$ (in the appropriate intervals) then the
third is uniquely determined by the congruence condition $i+j+k\equiv 0\pmod m$.  Therefore
the sum over $(i,j,k) \in {\mathcal M}$ above equals $S_0^2 +2S_0 S_1$,
and our inequality above reads
\begin{equation}
\label{eq:seq_sym_2}
S_0^2 +2 S_0 S_1 \le 3 (m^2-1) + (x_0^2 +2x_0 x_m) \le 3m^2 + (x_0^2 -x_m^2),
\end{equation}
where the final inequality follows since $x_m^2 +2 x_0 x_m \le 3$ using
\eqref{eq:seq_sym_1} with $(m,m,0)$.

Consider first the case $S_1 \le 0$.  Since $S_1 \ge -0.2 m$ we conclude from \eqref{eq:seq_sym_2} that
$$
(S_0+S_1)^2 \le 3m^2 + S_1^2 + (x_0^2 -x_m^2) \le 3m^2 + (0.2m)^2 + (2.2)^2 < 4m^2,
$$
since $m\ge 3$.  It follows that $S_0+S_1 \le 2m$, contradicting our assumption that $X>1$.

It remains to consider the case $S_1 >0$.  We must then have $x_m \ge 0$, and note
that $S_1 \le mx_m$.  Therefore \eqref{eq:seq_sym_2} yields
$$
4m^2 < (S_0+S_1)^2 \le 3m^2 + m^2 x_m^2 + x_0^2 -x_m^2,
$$
and upon rearranging we have
\begin{equation}
\label{eq:seq_sym_3}
x_0^2 > 1+ (m^2 -1) (1-x_m^2).
\end{equation}
But the condition $x_m^2 + 2x_0 x_m \le 3$ implies that
$x_0 x_m \le 1+ (1-x_m^2)/2$, and upon squaring that $x_0^2 x_m^2 \le 1+ (1-x_m^2) + (1-x_m^2)^2/4$.  Combining this with the lower bound \eqref{eq:seq_sym_3}, and note that $x_m^2\leq 1$ by \eqref{eq:seq_sym_1} with $(m,m,m)$, we conclude that
$x_m^2 < 9/(4m^2 -3)$.  but when this is entered into \eqref{eq:seq_sym_3} we obtain a
contradiction to $x_0\le 2.2$.
\end{proof}

We remark that the constant $5/8$ appearing in the statement can be replaced by an arbitrary $\alpha\in (0,1)$, with the condition $n\geq 6$ replaced by $n\geq N(\alpha)$ for some constant $N(\alpha)$ depending on $\alpha$. For a fixed $\alpha$, one can work out exactly what $N(\alpha)$ is, by following the above argument in the proof. We will only be concerned with the case $\alpha=5/8$, though.

Note that  Lemma \ref{lem:seq_sym} fails for $n=4$. For example, we may take $a_0=1$, $a_1=0.6$, $a_2=0.5$, and $a_3=0.41$.

\begin{lem}\label{lem:seq}
Let $n\geq 10$ be an even positive integer. Let $\{a_i\},\{b_i\},\{c_i\}$ ($0\leq i<n$) be three decreasing sequences of reals in $[0,1]$. Let $A,B,C$ denote the averages of $a_i,b_i,c_i$, respectively. Suppose that for all triples $(i,j,k)$ with $0\leq i,j,k<n$ and $i+j+k\geq n$ we have
\[ a_ib_j+b_jc_k+c_ka_i\leq\tfrac{5}{8}(a_i+b_j+c_k). \]
Then
\[ AB+BC+CA\leq\tfrac{5}{8}(A+B+C). \]
\end{lem}

This is simply an asymmetric version of Lemma \ref{lem:seq_sym}. It is needed to complete the induction process of proving Proposition \ref{prop:main}. The heart of the proof of Lemma \ref{lem:seq} is similar as that of Lemma \ref{lem:seq_sym}: we deduce from the hypotheses an inequality such as \eqref{eq:seq_sym_2}, and proceed from there by dividing into cases. Unfortunately, there are many more cases to consider in the asymmetric version.

\begin{proof}
As before, we make the change of variables $x_i=\frac{16}{5}a_i-1$, $y_i=\frac{16}{5}b_i-1$, and $z_i=\frac{16}{5}c_i-1$. Then $\{x_i\}$, $\{y_i\}$, and $\{z_i\}$ form decreasing sequences of reals in $[-1,2.2]$. By hypothesis, for all triples $(i,j,k)$ with $0\leq i,j,k<n$ and $i+j+k\geq n$,
\begin{equation}\label{eq:seq_1}
x_iy_j+y_jz_k+z_kx_i\leq 3.
\end{equation}
If $X,Y,Z$ denote the averages of $x_i,y_i,z_i$, respectively, then our goal becomes to show that
\[ XY+YZ+ZX\leq 3. \]

Write $n=2m$ and define
\[ X_0=\sum_{i=0}^{m-1}x_i,X_1=\sum_{i=m}^{n-1}x_i,Y_0=\sum_{i=0}^{m-1}y_i,Y_1=\sum_{i=m}^{n-1}y_i,Z_0=\sum_{i=0}^{m-1}z_i,Z_1=\sum_{i=m}^{n-1}z_i. \]
Let $\mathcal{M}$ denote the set of triples $(i,j,k)$ with $0\leq i,j\leq m-1, m\leq k\leq n-1$ and $i+j+k\equiv 0$ (mod $m$). Just as before, we sum the condition \eqref{eq:seq_1} over all triples $(i,j,k)$ in $\mathcal{M}$ except for $(0,0,m)$ to get
\[ \sum_{(i,j,k)\in\mathcal{M}}(x_iy_j+y_jz_k+z_kx_i)-(x_0y_0+y_0z_m+z_mx_0)\leq 3(m^2-1). \]
Again, if two of the variables $i,j,k$ are fixed, then the third is uniquely determined by the congruence condition $i+j+k\equiv 0$ (mod $m$). Hence the sum over $(i,j,k)\in\mathcal{M}$ above equals $X_0Y_0+Y_0Z_1+Z_1X_0$, and our inequality reads
\[ X_0Y_0+Y_0Z_1+Z_1X_0 \leq 3(m^2-1)+(x_0y_0+y_0z_m+z_mx_0). \]
Similarly, we have the other two inequalities
\[ X_0Y_1+Y_1Z_0+Z_0X_0 \leq 3(m^2-1)+(x_0y_m+y_mz_0+z_0x_0), \]
\[ X_1Y_0+Y_0Z_0+Z_0X_1 \leq 3(m^2-1)+(x_my_0+y_0z_0+z_0x_m). \]
Using the above three inequalities, it follows that
\begin{align*}
&\quad n^2(XY+YZ+ZX)  \\
&= (X_0+X_1)(Y_0+Y_1)+(Y_0+Y_1)(Z_0+Z_1)+(Z_0+Z_1)(X_0+X_1) \\
&\le 9(m^2-1)+(x_0y_0+y_0z_m+z_mx_0)+(x_0y_m+y_mz_0+z_0x_0) \\
&\quad +(x_my_0+y_0z_0+z_0x_m)+(X_1Y_1+Y_1Z_1+Z_1X_1) \\
&= 9(m^2-1)+U+V-W,
\end{align*}
where
\begin{align*}
U &=(x_0+x_m)(y_0+y_m)+(y_0+y_m)(z_0+z_m)+(z_0+z_m)(x_0+x_m), \\
V &= X_1Y_1+Y_1Z_1+Z_1X_1, \\
W &= x_my_m+y_mz_m+z_mx_m.
\end{align*}
In order to prove that $XY+YZ+ZX\leq 3$, it suffices to prove that
\begin{equation}\label{eq:seq_goal}
U+V-W\leq 3m^2+9.
\end{equation}

By summing over the condition \eqref{eq:seq_1} over all triples $(i,j,k)$ with $m\leq i,j,k<n$ and $i+j+k\equiv 0$ (mod $m$), except for $(m,m,m)$, we conclude that
\begin{equation}\label{eq:seq_2}
V-W\leq 3(m^2-1).
\end{equation}

Now we consider $U$. For convenience, write $r=x_0+x_m$, $s=y_0+y_m$, and $t=z_0+z_m$. Then $r,s,t\in [-2,4.4]$, and
\[ U=rs+st+tr. \]
Consider the three numbers $r+s$, $s+t$, and $t+r$. If at least one of them is negative, say $r+s<0$, then
\[ U=(r+s)t+rs\leq rs-2(r+s)=(r-2)(s-2)-4\leq (-4)\cdot (-4)-4=12, \]
and this gives \eqref{eq:seq_goal} upon using \eqref{eq:seq_2}. Hence we may assume that $r+s$, $s+t$, and $t+r$ are all nonnegative. This implies that $U$ is an increasing function of each of its variables $r$, $s$, and $t$.

We now bound $V$ in terms of $x_m$, $y_m$, and $z_m$. By the monotonicity of $\{x_i\}$, $\{y_i\}$, and $\{z_i\}$, together with the fact that $x_i,y_i,z_i\geq -1$, we have
\[ x_m-(m-1)\leq X_1\leq mx_m,\quad y_m-(m-1)\leq Y_1\leq my_m,\quad z_m-(m-1)\leq Z_1\leq mz_m. \]

We consider four cases, depending on the signs of $X_1$, $Y_1$, and $Z_1$.

\vspace{5 mm}
\textbf{Case (i)}: If $X_1,Y_1,Z_1<0$, then $V$ is decreasing in each of its variables $X_1,Y_1,Z_1$, and thus
\begin{align*}
V &\leq [x_m-(m-1)][y_m-(m-1)]+[y_m-(m-1)][z_m-(m-1)]+ \\
&\quad +[z_m-(m-1)][x_m-(m-1)] \\
&= 3(m-1)^2-2(m-1)(x_m+y_m+z_m)+W.
\end{align*}
We use the following trivial estimate for $U$ (recall that, under our assumption, $U$ is an increasing function of $r,s,t$):
\begin{align*}
U &\leq (2.2+x_m)(2.2+y_m)+(2.2+y_m)(2.2+z_m)+(2.2+z_m)(2.2+x_m) \\
&= 14.52+4.4(x_m+y_m+z_m)+W.
\end{align*}
Combining the above bounds for $U$ and $V$, we get
\begin{align*}
U+V-W &\leq 14.52+3(m-1)^2+(x_my_m+y_mz_m+z_mx_m)-(2m-6.4)(x_m+y_m+z_m) \\
&=14.52+3(m-1)^2-3(m-3.2)^2+(x_m-m+3.2)(y_m-m+3.2)+ \\
& +(y_m-m+3.2)(z_m-m+3.2)+(z_m-m+3.2)(x_m-m+3.2).
\end{align*}
Since $x_m-m+3.2,y_m-m+3.2,z_m-m+3.2\in [2.2-m,5.4-m]$ and $m\geq 4$, the right side above attains its maximum when $x_m=y_m=z_m=-1$. Hence,
\[ U+V-W\leq 3m^2+1.32, \]
thus confirming \eqref{eq:seq_goal}.

\vspace{5 mm}
\textbf{Case (ii)}: If exactly two of $X_1,Y_1,Z_1$ are negative, say $X_1,Y_1<0$ and $Z_1\geq 0$, then we have $V\leq X_1Y_1\leq m^2$ because $X_1,Y_1\in [-m,0)$. For $U-W$, after expanding out, we can write $U-W$ as sum of nine terms, and each term is bounded above by $2.2^2$. Hence $U-W\leq 9\cdot 2.2^2$. Combining the above bounds for $V$ and $U-W$, and recalling that $m\geq 5$, one gets \eqref{eq:seq_goal} after some simple computations.

\vspace{5 mm}
\textbf{Case (iii)}: If exactly one of $X_1,Y_1,Z_1$ is negative, say $X_1<0$ and $Y_1,Z_1\geq 0$, and either $X_1+Y_1$ or $X_1+Z_1$ is negative, say $X_1+Y_1<0$, then we have $V=(X_1+Y_1)Z_1+X_1Y_1\leq 0$. The rest follows as in Case (ii).

\vspace{5 mm}
\textbf{Case (iv)}: We are now left with the case that $X_1+Y_1$, $Y_1+Z_1$, and $X_1+Z_1$ are all non-negative. In particular, we have $x_m+y_m\geq 0$, $y_m+z_m\geq 0$, and $z_m+x_m\geq 0$. Under this assumption, $V$ is an increasing function of $X_1$, $Y_1$, and $Z_1$, and thus
\begin{equation}\label{eq:seq_4}
V\leq m^2(x_my_m+y_mz_m+z_mx_m)=m^2W.
\end{equation}
By \eqref{eq:seq_1} applied to $(0,m,m),(m,0,m),(m,m,0)$,
\[ x_0y_m+y_mz_m+z_mx_0\leq 3,\quad x_my_0+y_0z_m+z_mx_m\leq 3,\quad x_my_m+y_mz_0+z_0x_m\leq 3. \]
Summing the above three inequalities we get
\begin{equation}\label{eq:seq_3}
W+x_0y_m+x_0z_m+y_0x_m+y_0z_m+z_0x_m+z_0y_m\leq 9,
\end{equation}
and this is equivalent to
\begin{equation}\label{eq:seq_5}
U\leq 9+(x_0y_0+y_0z_0+z_0x_0).
\end{equation}

We first assume that
\[ x_m+y_m\geq 0.55,\quad y_m+z_m\geq 0.55,\quad z_m+x_m\geq 0.55. \]
This implies that $x_0+y_0-8(x_m+y_m)\leq 4.4-8\times 0.55=0$, and thus
\[ [x_0+y_0-8(x_m+y_m)](z_0-z_m)\leq 0, \]
and upon expanding we have
\[ (x_0z_0+y_0z_0)+8(x_mz_m+y_mz_m)\leq x_0z_m+y_0z_m+8(x_mz_0+y_mz_0). \]
Similarly, we have the other two inequalities
\[ (y_0x_0+z_0x_0)+8(y_mx_m+z_mx_m)\leq y_0x_m+z_0x_m+8(y_mx_0+z_mx_0), \]
\[ (z_0y_0+x_0y_0)+8(z_my_m+x_my_m)\leq z_0y_m+x_0y_m+8(z_my_0+x_my_0). \]
Suming the above three inequalities, we get
\[ 2(x_0y_0+y_0z_0+z_0x_0)+16W\leq 9(x_0y_m+x_0z_m+y_0x_m+y_0z_m+z_0x_m+z_0y_m), \]
and by \eqref{eq:seq_3}, we then have
\begin{equation}\label{eq:seq_6}
2(x_0y_0+y_0z_0+z_0x_0)+25W\leq 81.
\end{equation}
Combining \eqref{eq:seq_4}, \eqref{eq:seq_5}, and \eqref{eq:seq_6},
we get
\begin{align*}
U+V-W &\leq  9+(x_0y_0+y_0z_0+z_0x_0)+(m^2-1)W \\
&\leq 9+\tfrac{81}{2}+\left(m^2-\tfrac{27}{2}\right)W \\
&\leq 9+\tfrac{81}{2}+3\left(m^2-\tfrac{27}{2}\right),
\end{align*}
since $W\leq 3$ by \eqref{eq:seq_1} and $m\geq 5$. This gives \eqref{eq:seq_goal}.

It only remains to consider the case when at least one of $x_m+y_m$, $y_m+z_m$, and $z_m+x_m$ is less than $0.55$, say $x_m+y_m<0.55$. Since $x_m+y_m\geq 0$ under our assumption, we have $x_my_m<(0.55/2)^2$, and thus
\[ W=z_m(x_m+y_m)+x_my_m\leq 2.2\times 0.55+\left(\tfrac{0.55}{2}\right)^2<1.3. \]
Upon using \eqref{eq:seq_4} and \eqref{eq:seq_5}, we then have
\[ U+V-W\leq 9+(x_0y_0+y_0z_0+z_0x_0)+1.3(m^2-1). \]
Inserting the trivial bound $x_0y_0+y_0z_0+z_0x_0\leq 3\cdot 2.2^2$ above, we get \eqref{eq:seq_goal} when $m\geq 5$. This completes the proof.
\end{proof}

Numerical experiments show that both Lemma \ref{lem:seq_sym} and Lemma \ref{lem:seq} should be true for $n\geq 6$. However, in order to deal with the case $n=6$ in Lemma \ref{lem:seq}, the treatment of several cases seems to be more delicate. Fortunately, the combination of these two lemmas as stated would suffice for our purpose, and we thus do not pursue the optimal statement for Lemma \ref{lem:seq}.

Finally, we need a simple version of Corollay \ref{cor:main} when
$m=15$, which can be checked by a computer program.

\begin{lem}\label{lem:base}
Let $A,B,C\subset\Z_{15}^*$ be three arbitrary subsets. Suppose that
\[ |A||B|+|B||C|+|C||A|>5(|A|+|B|+|C|). \]
Then $A+B+C=\Z_{15}$.
\end{lem}

\begin{proof}
If $|A|+|B|+|C|>2\phi(15)=16$, then the conclusion follows by
Corollary 2.1 of \cite{li2007} (take $f_1,f_2,f_3$ there to be the
characteristic functions of $A,B,C$, respectively). Henceforth
assume that $|A|+|B|+|C|\leq 16$. Assume also that $|A|\leq |B|\leq
|C|$. There are only four possibilities for the triple
$(|A|,|B|,|C|)$:
\[ (3,6,7),\ \ (4,5,7),\ \ (4,6,6),\ \ (5,5,6). \]
These give fewer than $2\times 10^5$ possibilities for the sets
$A,B,C$. A computer program can easily do the job of checking all
possibilities.
\end{proof}

%%%%%%%%%%%%%%%%%%%%%%%%%%%%%%%%%%%%%%%%%%%%%%%%%%%%%%%%%
%%%%%%%%%%%%%%%%%%% local result %%%%%%%%%%%%%%%%%%%%%%%%
%%%%%%%%%%%%%%%%%%%%%%%%%%%%%%%%%%%%%%%%%%%%%%%%%%%%%%%%%

\section{Local Results}\label{sec:loc}

Proposition \ref{prop:main} follows from Proposition \ref{prop:5/8} and Proposition \ref{prop:15}.

\begin{prop}\label{prop:5/8}
Let $m$ be a squarefree positive integer with $(m,30)=1$. Let $f:\Z_m^*\rightarrow [0,1]$ be an arbitrary function with average larger than $5/8$:
\[ \sum_{x\in\Z_m^*}f(x)>\tfrac{5}{8}\phi(m). \]
Then for any $x\in\Z_m$, there exists $a,b,c\in\Z_m^*$ with $x=a+b+c$, such that
\begin{equation}\label{eq:prop5/8_conc}
f(a)f(b)+f(b)f(c)+f(c)f(a)>\tfrac{5}{8}(f(a)+f(b)+f(c)).
\end{equation}
\end{prop}

Note that the conclusion \eqref{eq:prop5/8_conc} above implies that
$f(a)f(b)f(c)>0$: if $f(a)=0$, for example, then
$\eqref{eq:prop5/8_conc}$ gives
\[ f(b)f(c)>\tfrac{5}{8}(f(b)+f(c)). \]
This is impossible since $0\leq f(b),f(c)\leq 1$.

\begin{proof}
We proceed by induction. First consider the base case when $m=p\geq 7$ is prime. Let $a_0\geq a_1\geq\cdots\geq a_{p-2}$ be the $p-1$ values of $f(x)$ ($x\in\F_p^*$) in decreasing order. The average of the sequence $\{a_i\}$ is larger than $5/8$ by hypothesis. By Lemma \ref{lem:seq_sym} applied to the sequence $\{a_i\}$ of length $p-1\geq 6$, we conclude that there exists $0\leq i,j,k\leq p-2$ with $i+j+k\geq p-1$ such that
\begin{equation}\label{eq:5/8_2}
a_ia_j+a_ja_k+a_ka_i>\tfrac{5}{8}(a_i+a_j+a_k).
\end{equation}
Define $I,J,K\subset\F_p^*$ by
\[ I=\{x:f(x)\geq a_i\},\quad J=\{x:f(x)\geq a_j\},\quad K=\{x:f(x)\geq a_k\}. \]
By the monotonicity of the sequence $\{a_i\}$,
\[ |I|+|J|+|K|\geq (i+1)+(j+1)+(k+1)\geq p+2, \]
and thus by the Cauchy-Davenport-Chowla theorem,
\[ I+J+K=\F_p. \]
For any $x\in\F_p$, we can thus find $u\in I,v\in J,w\in K$ with $x=u+v+w$. By the definitino of $I,J,K$, we have
\[ f(u)\geq a_i,\quad f(v)\geq a_j,\quad f(w)\geq a_k. \]
Observe that the function $h(x,y,z)=xy+yz+zx-\tfrac{5}{8}(x+y+z)$ is
increasing in each variable in the region
\[ \{0\leq x,y,z\leq
1:x+y\geq\tfrac{5}{8},y+z\geq\tfrac{5}{8},z+x\geq\tfrac{5}{8}\}. \]
Note that $\eqref{eq:5/8_2}$ implies $a_i+a_j\geq\tfrac{5}{8}$. This
is because, if $a_i+a_j<\tfrac{5}{8}$, then
$a_ia_k+a_ja_k\leq\tfrac{5}{8}a_k$, and thus
$a_ia_j>\tfrac{5}{8}(a_i+a_j)$, which is impossible as $0\leq
a_i,a_j\leq 1$. Similarly, $a_j+a_k\geq\tfrac{5}{8}$ and
$a_k+a_i\geq\tfrac{5}{8}$. The monotonicity of $h(x,y,z)$ together
with \eqref{eq:5/8_2} then implies
\[ f(u)f(v)+f(v)f(w)+f(w)f(u)>\tfrac{5}{8}(f(u)+f(v)+f(w)). \]

We now assume that $m$ is composite and write $m=m'p$ with $p\geq 11$. We use the canonical isomorphism $\Z_m\cong\Z_{m'}\times\F_p$ to define $f':\Z_{m'}^*\rightarrow [0,1]$ by
\[ f'(x)=\frac{1}{p-1}\sum_{y\in\F_p^*}f(x,y). \]
Then by induction hypothesis, for any $x\in\Z_{m'}$, there exists $a,b,c\in\Z_{m'}^*$ with $x=a+b+c$, such that
\[ f'(a)f'(b)+f'(b)f'(c)+f'(c)f'(a)>\tfrac{5}{8}(f'(a)+f'(b)+f'(c)). \]
Let $a_0\geq a_1\geq\cdots\geq a_{p-2}$ be the $p-1$ values of $f(a,x)$ ($x\in\F_p^*$) in decreasing order. Define similarly the sequences $\{b_i\}$ and $\{c_i\}$. The averages of the sequences $\{a_i\},\{b_i\},\{c_i\}$ are $f'(a),f'(b),f'(c)$, respectively. Hence by Lemma \ref{lem:seq} applied to the sequences $\{a_i\}$, $\{b_i\}$, and $\{c_i\}$ of length $p-1\geq 10$, we conclude that there exists $0\leq i,j,k\leq p-2$ with $i+j+k\geq p-1$, such that
\begin{equation}\label{eq:5/8_1}
a_ib_j+b_jc_k+c_ka_i>\tfrac{5}{8}(a_i+b_j+c_k).
\end{equation}
Define $I,J,K\subset\F_p^*$ by
\[ I=\{x:f(a,x)\geq a_i\},\quad J=\{x:f(b,x)\geq b_j\},\quad K=\{x:f(c,x)\geq c_k\}. \]
By the monotonicity of the sequences $\{a_i\}$, $\{b_i\}$, and $\{c_i\}$,
\[ |I|+|J|+|K|\geq (i+1)+(j+1)+(k+1)\geq p+2, \]
and thus by the Cauchy-Davenport-Chowla theorem,
\[ I+J+K=\F_p. \]
For any $y\in\F_p$, we can thus find $u\in I,v\in J,w\in K$ with $y=u+v+w$. By the definitions of $I,J,K$, we have
\[ f(a,u)\geq a_i,\quad f(b,v)\geq b_j,\quad f(c,w)\geq c_k. \]
It then follows easily from \eqref{eq:5/8_1} that
\[ f(a,u)f(b,v)+f(b,v)f(c,w)+f(c,w)f(a,u)>\tfrac{5}{8}(f(a,u)+f(b,v)+f(c,w)). \]
This completes the induction step.
\end{proof}

\begin{prop}\label{prop:15}
Let $f_1,f_2,f_3:\Z_{15}^*\rightarrow [0,1]$ be arbitrary functions. Let
\[ F_i=\sum_{x\in\Z_{15}^*}f_i(x). \]
Suppose that
\begin{equation} \label{eq:prop15_condition}
F_1F_2+F_2F_3+F_3F_1>5(F_1+F_2+F_3).
\end{equation}
Then for any $x\in\Z_{15}$, there exists $a_1,a_2,a_3\in\Z_{15}^*$
with $x=a_1+a_2+a_3$, such that
\[ f_1(a_1)f_2(a_2)f_3(a_3)>0,\quad f_1(a_1)+f_2(a_2)+f_3(a_3)>\tfrac{3}{2}. \]
\end{prop}

\begin{proof}
Our proof is based on case by case analysis. Let $A_i$ be the
support of $f_i$ ($A_i=\{x\in\Z_{15}^*:f_i(x)>0\}$) and write
$n_i=|A_i|$. Without loss of generality, we may assume that $n_1\leq
n_2\leq n_3$. Since $n_i\geq F_i$, we have by
\eqref{eq:prop15_condition},
\begin{equation}\label{eq:prop15_1}
n_1n_2+n_2n_3+n_3n_1>5(n_1+n_2+n_3).
\end{equation}
Let $M$ be the set of all possible $(n_1,n_2,n_3)$ satisfying
(\ref{eq:prop15_1}). Since $n_1,n_2,n_3\in\{0,1,\ldots,8\}$, it
turns out that $|M|=34$. Fix some $(n_1,n_2,n_3)\in M$. For
$i\in\{1,2,3\}$, write $A_i=\{x_{i1},x_{i2},\cdots,x_{in_i}\}$. Let
$y_{ij}=f(x_{ij})$ ($1\leq j\leq n_i$). Without loss of generality,
assume that $y_{i1}\geq y_{i2}\geq\cdots\geq y_{in_i}$.

Let $J$ be the set of all triples $(j_1,j_2,j_3)$ satisfying
\begin{equation}
\label{eq:prop15_2} 1\leq j_i\leq n_i,\quad
j_1j_2+j_2j_3+j_3j_1>5(j_1+j_2+j_3).
\end{equation}
First assume that $y_{1j_1}+y_{2j_2}+y_{3j_3}>3/2$ for some
$(j_1,j_2,j_3)\in J$. Then by Lemma \ref{lem:base} applied to the
three sets $\{x_{i1},\cdots,x_{ij_i}\}$ ($i=1,2,3$), there exists
$a_i\in\{x_{i1},\cdots,x_{ij_i}\}$ such that $x=a_1+a_2+a_3$. The
proof is then complete because \[ f_1(a_1)+f_2(a_2)+f_3(a_3)\geq
y_{1j_1}+y_{2j_2}+y_{3j_3}>\tfrac{3}{2}. \]

Henceforth assume that
\begin{equation}\label{eq:prop15_3}
y_{1j_1}+y_{2j_2}+y_{3j_3}\leq\tfrac{3}{2}
\end{equation}
for all $(j_1,j_2,j_3)\in J$. Set up an optimization problem with
variables $y_{ij}\in [0,1]$ ($1\leq j\leq n_i$) and constraints
(\ref{eq:prop15_3}). Our objective is to maximize the sum of all
variables:
\[ S=\sum_{i=1}^3\sum_{j=1}^{n_i}y_{ij}=F_1+F_2+F_3. \]
The constraints and the objective function in this optimization problem are all linear, and the maximum of $S$ can be found using a linear programming algorithm. Our conclusion is the following:
\[ F_1+F_2+F_3\leq\begin{cases} 16 & (n_1,n_2,n_3)=(2,8,8), \\ 15.5 & (n_1,n_2,n_3)\in M_1, \\ 15 & \text{otherwise,} \end{cases} \]
where
\[ M_1=\{(2,7,8),(3,6,8),(3,8,8),(4,5,8),(4,8,8)\}. \]
Hence, for $(n_1,n_2,n_3)\notin M_1$ and $(n_1,n_2,n_3)\neq
(2,8,8)$, we have $F_1+F_2+F_3\leq 15$, from which one easily gets a
contradiction with (\ref{eq:prop15_condition}):
\begin{equation}\label{eq:prop15_contra}
F_1F_2+F_2F_3+F_3F_1\leq\tfrac{1}{3}(F_1+F_2+F_3)^2\leq 5(F_1+F_2+F_3).
\end{equation}

Now we consider the remaining cases. For notational convenience, we shall write
\[ T(x,y,z)=xy+yz+zx-5(x+y+z). \]
By \eqref{eq:prop15_condition}, $T(F_1,F_2,F_3)>0$, and it is easy to see that $T(x,y,z)$ is an increasing function of $x,y,z$ in the range $x\geq F_1$, $y\geq F_2$, $z\geq F_3$. We aim for a contradiction in each case.

For $(n_1,n_2,n_3)\in\{(2,7,8),(2,8,8)\}$, we have $F_1+F_2+F_3\leq
16$ and $F_1\leq 2$, and thus
\[ T(F_1,F_2,F_3)\leq T(2,7,7)<0. \]

For $(n_1,n_2,n_3)\in\{(3,6,8),(3,8,8)\}$, we have $F_1+F_2+F_3\leq
15.5$ and $F_1\leq 3$, and thus
\[ T(F_1,F_2,F_3)\leq T(3,6.25,6.25)<0. \]

For $(n_1,n_2,n_3)=(4,5,8)$, we divide into two possibilities:
$F_1+F_2\geq 8$ and $F_1+F_2\leq 8$. If $F_1+F_2\geq 8$, we solve
the optimization problem as before, but with the additional
constraint given by $F_1+F_2\geq 8$:
\[ \sum_{i=1}^2\sum_{j=1}^{n_i}y_{ij}\geq 8, \]
to conclude that $S\leq 15$, and we are done by \eqref{eq:prop15_contra}. If $F_1+F_2\leq 8$, we then have
\[ T(F_1,F_2,F_3)\leq T(4,4,7.5)<0. \]

Finally, for $(n_1,n_2,n_3)=(4,8,8)$, we again divide into two
possibilities: $F_1\geq 3$ and $F_1\leq 3$. If $F_1\geq 3$, we solve
the optimization problem as before, but with the additional
constraint given by $F_1\geq 3$:
\[ \sum_{j=1}^{n_1}y_{1j}\geq 3, \]
to conclude that $S\leq 15$, and we are done by \eqref{eq:prop15_contra}. If $F_1\leq 3$, then, as in the case $(a_1,a_2,a_3)\in\{(3,6,8),(3,8,8)\}$, we have
\[ T(F_1,F_2,F_3)\leq T(3,6.25,6.25)<0. \]
This completes the proof.
\end{proof}

\begin{proof}[Proof of Proposition \ref{prop:main}]
First note that if the result holds for $m$, then it also holds for any $m'$ dividing $m$. Hence we may assume that $15|m$. Write $m=15m'$, where $(m',30)=1$. We shall use the canonical isomorphism $\Z_m\cong\Z_{m'}\times\Z_{15}$. Let $(u,v)\in\Z_m$ be arbitrary ($u\in\Z_{m'}$, $v\in\Z_{15}$). Define $f':\Z_{m'}\rightarrow [0,1]$ by
\[ f'(x)=\frac{1}{\phi(15)}\sum_{y\in\Z_{15}^*}f(x,y). \]

Since the average of $f$ is larger than $5/8$, so is the average of $f'$. We may apply Proposition \ref{prop:5/8} to $f'$, concluding that there exists $a_1,a_2,a_3\in\Z_{m'}^*$ with $u=a_1+a_2+a_3$, such that
\begin{equation}\label{eq:prop_main_1}
f'(a_1)f'(a_2)+f'(a_2)f'(a_3)+f'(a_3)f'(a_1)>\tfrac{5}{8}(f'(a_1)+f'(a_2)+f'(a_3)).
\end{equation}
Now define $f_i:\Z_{15}^*\rightarrow [0,1]$ ($i=1,2,3$) by
\[ f_i(y)=f(a_i,y). \]
Note that (\ref{eq:prop_main_1}) implies that the hypothesis of Proposition \ref{prop:15} holds for $f_1,f_2,f_3$, and we can conclude that there exists $b_1,b_2,b_3\in\Z_{15}^*$ with $v=b_1+b_2+b_3$, such that
\[ f_i(b_i)>0,\quad f_1(b_1)+f_2(b_2)+f_3(b_3)>\tfrac{3}{2}. \]
Equivalently, we have $(u,v)=(a_1,b_1)+(a_2,b_2)+(a_3,b_3)$, and
\[ f(a_i,b_i)>0,\quad f(a_1,b_1)+f(a_2,b_2)+f(a_3,b_3)>\tfrac{3}{2}. \]
This completes the proof.
\end{proof}

%%%%%%%%%%%%%%%%%%%%%%%%%%%%%%%%%%%%%%%%%%%%%%%%%%%%%%%%
%%%%%%%%%%%%% transference principle %%%%%%%%%%%%%%%%%%%
%%%%%%%%%%%%%%%%%%%%%%%%%%%%%%%%%%%%%%%%%%%%%%%%%%%%%%%%

\section{Proof of Theorem \ref{thm:main} and Theorem \ref{thm:main2}}\label{sec:trans}

\subsection{The transference principle}

In this section, we work in $\Z_N=\Z/N\Z$. For a function $f:\Z_N\rightarrow\C$, its Fourier transform is defined by
\[ \hat{f}(r)=\sum_{x\in\Z_N}f(x)e_N(rx),\quad r\in\Z_N, \]
where $e_N(y)=\exp(2\pi iy/N)$ as usual.

The convolution of two functions $f,g:\Z_N\rightarrow\C$ is defined by
\[ f*g(x)=\sum_{y\in\Z_N}f(y)g(x-y) \]
for $x\in\Z_N$.

\begin{prop}\label{prop:trans}
Let $N$ be a sufficiently large prime. Suppose that $\mu_i:\Z_N\rightarrow\R^+$ and $a_i:\Z_N\rightarrow\R^+$ ($i=1,2,3$) are functions satisfying the majorization condition
\begin{equation}\label{eq:major}
0\leq a_i(n)\leq\mu_i(n),
\end{equation}
and the mean condition
\begin{equation}\label{eq:mean}
\min(\delta_1,\delta_2,\delta_3,\delta_1+\delta_2+\delta_3-1)\geq\delta
\end{equation}
for some $\delta>0$, where $\delta_i=\sum_{x\in\Z_N}a_i(x)$. Suppose
that $\mu_i$ and $a_i$ also satisfiy the pseudorandomness conditions
\begin{equation}\label{eq:random_mu}
|\hat{\mu_i}(r)-\delta_{r,0}|\leq\eta
\end{equation}
for all $r\in\Z_N$, where $\delta_{r,0}$ is the Kronecker delta, and
\begin{equation}\label{eq:random_a}
\|\hat{a_i}\|_q=\left(\sum_{r\in\Z_N}|\hat{a_i}(r)|^q\right)^{1/q}\leq
M
\end{equation}
for some $2<q<3$ and $\eta,M>0$. Then for any $x\in\Z_N$, we have
\begin{equation}\label{eq:conclusion}
\sum_{y,z\in\Z_N}a_1(y)a_2(z)a_3(x-y-z)\geq \frac{c(\delta)}{N}
\end{equation}
for some constant $c(\delta)>0$ depending only on $\delta$, provided
that $\eta\leq\eta(\delta,M,q)$ is small enough.
\end{prop}

We first outline the proof. The first ingredient is to write $a_i$ as the sum $a_i'+a_i''$, where the function $a_i'$ is set-like, in the sense that it resembles the normalized characteristic function of a dense subset of $\Z_N$ (with density $\delta_i$), and the function $f_i''$ is uniform, in the sense that it has small Fourier coefficients. The second ingredient is to write the quantity in (\ref{eq:conclusion}) in terms of the Fourier coefficients of $a_i$, and to prove that one can replace $a_i$ by $a_i'$ at the cost of a small error term. The last step is to prove the corresponding result for $a_i'$. As $a_i'$ is set-like, this is essentially (a quantitative version of) the Cauchy-Davenport theorem.

We now carry out the details. We construct $a_i'$ as follows. Let
$\epsilon=\epsilon(\delta,M,q)>0$ be a small parameter to be chosen
later. Define
\[ R=\{r\in\Z_{N}:|\hat{a_i}(r)|\geq\epsilon\}. \]
Define a Bohr set $B$ by
\[ B=\{x\in\Z_{N}:|e_{N}(xr)-1|\leq\epsilon,r\in R\}. \]
Let $\beta$ be the normalized characteristic function of $B$:
\[ \beta=\frac{\mathbf{1}_B}{|B|}. \]
Let
\[ a_i'=a_i*\beta*\beta,\quad a_i''=a_i-a_i'. \]
We record the desired properties of $a_i'$ and $a_i''$.

\begin{lem}\label{lem:decomposition}
Suppose that the function $a_i$ satisfies the hypotheses in
Proposition \ref{prop:trans} and the functions $a_i'$ and $a_i''$
are constructed as above. Then,
\begin{enumerate}
\item $\sum_{x\in\Z_N}a_i'(x)=\delta_i$.
\item ($a_i'$ is set-like) $0\leq a_i'\leq (1+O_{\epsilon,M,q}(\eta))/N$.
\item ($a_i''$ is uniform) $\|\hat{a_i''}\|_{\infty}\ll\epsilon$.
\item $\|\hat{a_i'}\|_q\leq M$ and $\|\hat{a_i''}\|_q\leq M$.
\end{enumerate}
\end{lem}

\begin{proof}
See Proposition 5.1 of \cite{green2006}.
\end{proof}

We now consider the quantity in (\ref{eq:conclusion}), and show that the function $a_i$ there can be replaced by $a_i'$ at the cost of a small error term.

\begin{lem}\label{lem:error}
With $a_i'$ defined as above, we have, for any $x\in\Z_{N}$,
\[ \left|\sum_{y,z}a_i(y)a_i(z)a_i(x-y-z)-\sum_{y,z}a_i'(y)a_i'(z)a_i'(x-y-z)\right|\ll\frac{M^q\epsilon^{3-q}}{N}. \]
\end{lem}

\begin{proof}
Since $a_i=a_i'+a_i''$, the left side above is bounded by the sum of seven terms of the form
\[ S(f,g,h)=\left|\sum_{y,z}f(y)g(z)h(x-y-z)\right|, \]
where $f,g,h\in\{a_i',a_i''\}$ and $(f,g,h)\neq (a_i',a_i',a_i')$. Without loss of generality, we may assume that $h=a_i''$. We have
\[ S(f,g,h)=\frac{1}{N}\left|\sum_r\hat{f}(r)\hat{g}(r)\hat{h}(r)e_N(rx)\right|\leq\frac{1}{N}\sum_r|\hat{f}(r)\hat{g}(r)\hat{h}(r)|. \]
By Lemma \ref{lem:decomposition} (3), we have $|\hat{h}(r)|\ll\epsilon$. Hence,
\[ S(f,g,h)\ll N^{-1}\epsilon^{3-q}\sum_r|\hat{f}(r)\hat{g}(r)\hat{h}(r)^{q-2}|. \]
Now by H\"{o}lder's inequality and Lemma \ref{lem:decomposition} (4),
\[ S(f,g,h)\ll N^{-1}\epsilon^{3-q}\|\hat{f}\|_q\|\hat{g}\|_q\|\hat{h}\|_q^{q-2}\ll N^{-1}M^q\epsilon^{3-q}. \]
\end{proof}

We now treat the set-like function $a_i'$. Let
\[ X_i=\{x\in\Z_N:a_i'(x)\geq \delta^2/N\}. \]
Then
\[ \delta_i=\sum_{x\in\Z_N}a_i'(x)\leq (1+O_{\epsilon,M,q}(\eta))\frac{|X_i|}{N}+\delta^2. \]
Hence,
\[ |X_i|\geq (1-O_{\epsilon,M,q}(\eta))(\delta_i-\delta^2)N. \]
For $\eta$ sufficiently small depending on $\epsilon,\delta,M,q$,
the $O_{\epsilon,M,q}(\eta)$ term above is at most $\delta/8$ (say),
and thus for $N$ sufficiently large,
\[ \min(|X_1|,|X_2|,|X_3|,|X_1|+|X_2|+|X_3|-N)\geq\tfrac{\delta N}{2}. \]
The problem is now reduced to proving that $X_1+X_2+X_3=\Z_N$, and that each element of $\Z_N$ can be written in many ways as $x_1+x_2+x_3$ ($x_i\in X_i$).

\begin{lem}\label{lem:set}
Let $X_1,X_2,X_3\subset\Z_N$. Let $\theta_i=|X_i|/N$. Suppose that
\[ \theta=\min(\theta_1,\theta_2,\theta_3,\theta_1+\theta_2+\theta_3-1)>0. \]
Theree exists $N(\theta)$ and $c(\theta)>0$, such that if $N\geq N(\theta)$, then for any $x\in\Z_N$, there are at least $c(\theta)N^2$ ways to write $x=x_1+x_2+x_3$ with $x_i\in X_i$.
\end{lem}

\begin{proof}
See Lemma 3.3 of \cite{li2007}.
\end{proof}

It follows that for any $x\in\Z_N$,
\[ \sum_{y,z}a_i'(y)a_i'(z)a_i'(x-y-z)\geq\frac{\delta^6}{N^3}\sum_{y\in X_1,z\in X_2,x-y-z\in X_3}1\geq\frac{c(\delta)}{N}, \]
where $c(\delta)>0$ depends only on $\delta$. This proves
\eqref{eq:conclusion} in view of Lemma \ref{lem:error}, as long as
we choose $\epsilon$ small enough depending on $\delta,M,q$.

%%%%%%%%%%%%%%%%%%%%%%%%%%%%%% W-trick and stuff %%%%%%%%%%%%%%%%%%%%%%%%%%%%%%

\subsection{Deduction of Theorem \ref{thm:main} and Theorem \ref{thm:main2}}

We will prove these two theorems in parallel. Let $n$ be a
sufficiently large odd positive integer. We seek for a
representation $n=a_1+a_2+a_3$ with $a_1,a_2,a_3\in A$.

For Theorem \ref{thm:main}, let $\delta>0$ be a positive constant
such that \begin{equation}\label{eq:lower} \left|A\cap
\left[1,N\right]\right|>\left(\frac{5}{8}+\delta\right)\frac{N}{\log
N}
\end{equation}
for all sufficiently large $N$. For Theorem \ref{thm:main2}, let
$\delta>0$ be the constant as given in the statement.

In order for the transference principle to be applicable, one needs
to find a pseudorandom majorant $\mu$ for the characteristic
function of $A$. The characteristic function of the primes is not
pseudorandom because the primes are not equidistributed in residue
classes with small modulus. To remove this issue, we use the
``W-trick'' \cite{green2005,green-tao2}.

Let $z=z(\delta)$ be a large parameter to be chosen later and set
$W=\prod_{p\leq z}p$, where the product is over primes. Later we
will see that the larger $z$ is chosen to be, the more pseudorandom
our majorant will be (Lemma \ref{lem:psd}). For now, define a
function $f:\Z_W^*\rightarrow [0,1]$ by
\[ f(b)=\max\left(\frac{3\phi(W)}{2n}\sum_{x\in A\cap (W\Z+b),x<2n/3}\log x-\frac{\delta}{8},0\right). \]
Hence $0\leq f(b)\leq 1$ for every $b$. The prime number theorem in
arithmetic progressions implies the following. In the case of
Theorem \ref{thm:main}, by \eqref{eq:lower} we get
\begin{equation}\label{eq:cond1}
\sum_{b\in\Z_W^*}f(b)>\tfrac{5}{8}\phi(W).
\end{equation}
In the case of Theorem \ref{thm:main2}, by the hypotheses
\eqref{eq:mod3} and \eqref{eq:modW} we get
\begin{equation}\label{eq:cond2}
2\alpha_1+\alpha_2>\tfrac{3}{2},\ \
2\alpha_2+\alpha_1>\tfrac{3}{2},\ \ f(b)>0\text{ for every }b,
\end{equation}
where $\alpha_i$ ($i=1,2$) is the average of $f$ over those reduced
residues $b\pmod W$ with $b\equiv i\pmod 3$.

Write $m=W/2$ and note that $\Z_W^*\cong\Z_m^*$. If we work modulo
$m$, then Proposition \ref{prop:main} together with \eqref{eq:cond1}
or Proposition \ref{prop:main2} together with \eqref{eq:cond2}
implies that there exists $b_1,b_2,b_3\in\Z_W^*$ with
$b_1+b_2+b_3\equiv n\pmod m$ such that
\begin{equation}\label{eq:proof_1}
f(b_1)f(b_2)f(b_3)>0,\quad f(b_1)+f(b_2)+f(b_3)>\tfrac{3}{2}.
\end{equation}
Since $n$ is odd, we also know that $b_1+b_2+b_3\equiv n\pmod 2$,
and hence $b_1+b_2+b_3\equiv n\pmod W$.

The rest of the arguments will work simultaneously for both Theorem
\ref{thm:main} and Theorem \ref{thm:main2}. Let
\[ A_i=\left\{\frac{x-b_i}{W}:x\in A\cap (W\Z+b_i),x<\frac{2n}{3}\right\}. \]
Our goal becomes to prove that
\[ \frac{n-b_1-b_2-b_3}{W}\in A_1+A_2+A_3. \]
We pick a prime $N\in [(1+\kappa)n/W,(1+2\kappa)n/W]$, where $\kappa>0$ is small. By the choice of $N$, it is easy to see that it sufficies to show that $(n-b_1-b_2-b_3)/W\in A_1+A_2+A_3$ when $A_1,A_2,A_3$ are considered as subsets of $\Z_N$. From now on, we view $A_1,A_2,A_3$ as subsets of $\Z_{N}$, and we shall prove that
\[ A_1+A_2+A_3=\Z_N. \]
We have now completed the W-trick: there are no local obstructions for $A_1,A_2,A_3$ to be pseudorandom.

For $x\in [0,N-1]$, define
\[ \mu_i(x)=\begin{cases} \phi(W)\log(Wx+b_i)/(WN) &  \text{if }Wx+b_i\text{ is prime,} \\ 0 & \text{otherwise.} \end{cases} \]
Let $ a_i(x)=1_{A_i}(x)\mu_i(x)$. We then have the majorization condition
\[ 0\leq a_i(x)\leq\mu_i(x). \]
The sum of $a_i(x)$ is
\[ \delta_i=\sum_{x=0}^{N-1}a_i(x)=\frac{\phi(W)}{WN}\sum_{x\in A\cap(W\Z+b_i),x<2n/3}\log x=\frac{2n}{3WN}\left(f(b_i)+\frac{\delta}{8}\right). \]
By choosing $\kappa=\kappa(\delta)$ small enough, we then have
\[ \delta_i\geq\tfrac{2}{3}f(b_i)+\tfrac{\delta}{20}. \]
Hence by \eqref{eq:proof_1},
\[ \delta_i\geq\tfrac{\delta}{20},\quad \delta_1+\delta_2+\delta_3\geq 1+\tfrac{\delta}{20}, \]
This confirms the mean condition \eqref{eq:mean}. Recall that
$W=\prod_{p\leq z}p$.

\begin{lem}[Pseudorandom majorant]\label{lem:psd}
Suppose that $N$ and $z$ are sufficiently large. Then
\[ \sup_{r\neq 0}|\hat{\mu_i}(r)|\leq\frac{2\log\log z}{z}. \]
\end{lem}

\begin{proof}
See Lemma 6.2 of \cite{green2005}.
\end{proof}

This gives \eqref{eq:random_mu} when $r\neq 0$. The case $r=0$ follows from the prime number theorem.

\begin{lem}[Discrete majorant property]
Suppose that $q>2$. Then there is an absolute constant $C(q)$ such that
\[ \sum_r|\hat{a_i}(r)|^q\leq C(q). \]
\end{lem}

\begin{proof}
See Lemma 6.6 of \cite{green2005}.
\end{proof}

If $z=z(\delta)$ is chosen large enough, then Proposition
\ref{prop:trans} applies to give
\[ \sum_{y,z\in\Z_N}a_1(y)a_2(z)a_3(x-y-z)>0. \]
This shows that $A_1+A_2+A_3=\Z_N$, thus completing the proof of
Theorem \ref{thm:main} and Theorem \ref{thm:main2}.

We end this paper with a final remark. In the statement of Theorem
\ref{thm:main2}, we get the threshold density $1/2$ because we
artificially reduced the density by a factor of $2/3$ so that
equality in $\Z_N$ implies equality in $\Z$. The fact that $1/2$ is
a natural barrier may be related to the difficulty of proving
Vinogradov's theorem using purely sieve theory (without injecting
additional ingredients), due to the {\em parity phenomenon} in
analytic number theory. One way to state the parity problem is as
follows. It is very difficult to define ``almost primes" (or, more
precisely, weight functions) with good properties coming from sieve
theory in such a way that the density of the primes in the almost
primes gets larger than $1/2$. For an excellent account of sieve
theory including the parity phenomenon, see the book \cite{opera}.

\bibliographystyle{plain}
\bibliography{density_vinogradov}{}

\end{document}